\documentclass[a4paper,12pt,intlimits,oneside]{amsart}
\usepackage{graphicx}
\usepackage{enumerate}
\usepackage{amsfonts}
\usepackage{amsmath,amsthm,amssymb}
\usepackage{float}
\newcommand{\comment}[1]{}

\newtheorem{theorem}{Theorem}

\newtheorem{lemma}{Lemma}

\newcommand{\ffi}{\varphi}

\newcommand\ve{\varepsilon}

\newcommand\de{\delta}

\newcommand{\NN}{\mathbb N}
\newcommand{\ZZ}{\mathbb Z}
\newcommand{\RR}{\mathbb R}
\newcommand{\CC}{\mathbb C}
\newcommand{\TT}{\mathbb T}

\begin{document}
\title[Integral estimates]{On integral estimates of non-negative positive definite functions}
\author{}
\author{Andrey Efimov, Marcell Ga\'al and Szil\'ard Gy. R\'ev\'esz}
\address
{Andrey Efimov \newline \indent Ural Federal University \newline
\indent 620000 Ekaterinburg, pr. Lenina 51., RUSSIA}
\email{anothar@ya.ru}

\address{Marcell Ga\'al \newline  \indent Bolyai Institute, University of Szeged \newline  \indent 6720 Szeged, Aradi v\'ertan\'uk tere 1., HUNGARY \newline  \indent and \newline  \indent MTA-DE "Lend\"ulet" Functional Analysis Research Group \newline \indent Institute of Mathematics, University of Debrecen \newline \indent 4010 Debrecen, PO. Box 12, HUNGARY} \email{marcell.gaal.91@gmail.com}

\address{Szil\'ard Gy. R\'ev\'esz \newline  \indent Institute of Mathematics, Faculty of Sciences \newline \indent Budapest University of Technology and Economics \newline  \indent 1111 Budapest, M\H uegyetem rkp. 3-9., HUNGARY \newline  \indent and \newline  \indent A. R\'enyi Institute of Mathematics \newline \indent Hungarian Academy of Sciences, \newline \indent 1053 Budapest, Re\'Altanoda utca 13-15., HUNGARY} \email{revesz.szilard@renyi.mta.hu}

\begin{abstract} Let $\ell>0$ be arbitrary. We introduce the extremal quantities
\[
G(\ell):=\sup_{f} \dfrac{\int_{-\ell}^{\ell} f\,dx}{\int_{-1}^1 f\,dx},\quad
C(\ell):=\sup_{f} \sup_{a\in\RR} \dfrac{\int_{a-\ell}^{a+\ell} f\,dx}{\int_{-1}^1 f\,dx},
\]
where the supremum is taken over all not identically zero non-negative
positive definite functions. We are interested in the question
\textit{how large can the above extremal quantities be}? This problem
was originally posed by Yu. Shteinikov and S. Konyagin for the case
$\ell=2$. In this note we obtain exact values for the right limits
$\overline{\lim}_{\varepsilon\rightarrow 0+}G(k+\varepsilon)$ and
$\overline{\lim}_{\varepsilon\rightarrow 0+}C(k+\varepsilon)$ $(k\in
\NN)$, and sufficiently close bounds for other values of $\ell$. We
point out that the problem provides an extension of the classical
problem of Wiener.
\end{abstract}

\maketitle

\let\oldfootnote\thefootnote
\def\thefootnote{}
\footnotetext{} \footnotetext{This work was supported by the Program for State Support of Leading Scientific Schools of the Russian Federation (project no. NSh-9356.2016.1), by the Competitiveness Enhancement Program of the Ural Federal University (Enactment of the Government of the Russian Federation of March 16, 2013 no. 211, agreement no. 02.A03.21.0006 of August 27, 2013), and by Hungarian Science Foundation Grant \#'s NK-104183, K-109789. The second author was supported by the "Lend\" ulet" Program (LP2012-46/2012) of the Hungarian Academy of Sciences and the National Research, Development and Innovation Office -- NKFIH Reg. No. K115383.}
\let\thefootnote\oldfootnote

\bigskip
\bigskip



\bigskip
\bigskip
	
\section{Introduction}

At first let us fix some notation and basic concepts which will be used
throughout the sequel. The symbols $[\cdot]$ and $\lceil \cdot \rceil$
stand for the lower (or in other words, the usual) and the upper
integer part, respectively. We recall that a function $f:\RR \to \CC$
satisfying
\[
\sum_{i=1}^n\sum_{j=1}^n c_i\overline{c_j}f(x_i-x_j)\geq 0
\]
for any $n$-tuples $(x_i)_{i=1}^{n}\in\RR$ and complex numbers $(c_i)_{i=1}^{n}$ is called positive definite.
Positive definiteness of the function $f$ will be
denoted by $f\gg 0$. If the function $f$ is positive
definite and non-negative in the ordinary sense, then we
will say that $f$ is \emph{doubly positive} which we will also write as $f\ggg 0$. In what follows, the symbol $\star$ stands for the convolution.

The problem, formulated above, was originally posed by S. Konyagin and
Yu. Shteinikov, who wanted to use the estimate (in the case $\ell=2$)
for the paper \cite{Sht} dealing with number theory. In all our
investigations we take the liberty to discuss only the continuous case.
We believe that the transfer between the discrete and continuous
settings should not cause any difficulty.

We remark that the problem in question is closely related to the
celebrated Wiener's problem~\cite{Sh, Wiener}, i.e., the question if in
any $L^p$ norm and for any fixed $\delta>0$ the ratio
\[
\frac{\int_{-\pi}^{\pi}f^p(x)\,dx}{\int_{-\delta}^{\delta}f^p(x)\,dx}
\]
is bounded for all $2\pi$ periodical positive definite functions $f$.
Clearly, in an appropriate sense this holds for $p=\infty$, as we have
$\|f\|_\infty=f(0)$ for any positive definite function. Also, it can be
proved by means of the Parseval identity that for any given $\delta>0$,
we have
\[
\int_{-\pi}^{\pi} f^2\,dx \le \frac{2\pi}{\delta} \int_{-\delta}^{\delta} f^2\,dx
\]
whenever $f\in L^{\infty}$, say \cite{Sh, Wiener}. This, of course,
extends to any even $p=2m$ powers since if $f$ is positive definite,
then so is $f^m$. However, it is known that for no other exponents
$p\not\in 2\NN$ does such a finite bound hold. The first
counterexamples were constructed by Wainger \cite{Wa}, and the
strongest ones (with arbitrarily large gaps and only idempotent
polynomials in place of $f\in L^p$) can be found in \cite{BR}.

Furthermore, on the non-compact case of $\RR$, any bound between
integrals on $[-1,1]$ and $[-k,k]$ must grow to infinity with the
length $k$ as $\delta$ is fixed normalized to 1. This is explained in
\cite{GT} as "Wiener's property fails with $k\to\infty$". However, the
case is similar for $\delta\to 0$ in the torus $\TT:=\RR/\ZZ$ and on
the real line $\RR$ with $\de:=1$ and $k\to\infty$. The ratio in the
estimate must depend on the ratio of the corresponding intervals. In
this sense, both $\TT$ and $\RR$ behaves the same: there is a finite
upper bound exactly for $p\in 2\NN$ which bound happens to be linear in
the ratio of the compared intervals.

At this point let us note that the brave question under study is boldly
extending the classical Wiener's problem to the case of $L^1$ where it
is known to fail in general. The price we pay is that we restrict to
doubly positive functions instead of general positive definite
functions. However, this is in fact not a restriction but a
\emph{generalization}. Indeed, for any power $p=2m$, where Wiener's
problem has a positive answer, an estimate can be easily deducted from
the current setting if we observe the following: for any $f\gg 0$,
trivially $f^{2m}\ge 0$ and also by Schur's theorem $f^{2m}\gg 0$
whence $f^{2m}\ggg 0$. Thus the $L^1$-problem of Konyagin and
Shteinikov can be applied to deduce an answer to Wiener's problem even
if there is no Parseval identity at our help in this approach. In other
words, the positive answer in the question of Konyagin and Shteinikov
sheds light to the fact that somehow the positive cases of Wiener's
problem are not so intimately connected to Parseval's formula, while
the key now seems to be more of double positivity than any identity.

Konyagin's and Shteinikov's original question was answered positively
in~\cite{Gb} where Gorbachev found the following bound. (In fact, this
result was originally formulated for the discrete case.)

\begin{theorem}[D. V. Gorbachev]\label{pr:Gorbi} For any $f\ggg 0$ and $L>0$ we have
\[
\int_{-2L}^{2L} F dx\le \pi^2 \int_{-L}^{L} F\,dx.
\]
\end{theorem}

In terms of $G(k)$, this result can be reformulated as $G(2)\leq \pi^2$.
By iterating the above estimate one can obtain some bound for all interval length
ratio $\ell$. It happens to be not linear, however, a linear
growth might be expected in virtue of the known results in Wiener's problem.

In what follows, we obtain bounds for the whole range of $\ell$ which will be of linear
growth. This indeed allows us a direct derivation of the
positive answers in Wiener's problem when the exponent is $p=2m$. For
the case $k=2$, our upper bound is $G(2)\leq 5$ which is somewhat
better than $\pi^2$.

\section{The result}

Let $G(k+0)$ and $C(k+0)$ be the right limits
$\overline{\lim}_{\varepsilon\rightarrow 0+}G(k+\varepsilon)$ and
$\overline{\lim}_{\varepsilon\rightarrow 0+}C(k+\varepsilon)$,
respectively. We note that both functions $G(\ell)$ and $C(\ell)$ are
non-decreasing in $(0,\infty)$, and that $G(\ell)\leq C(\ell)$; also,
$C(\ell)=G(\ell)=1$ on $(0,1]$. Our result\footnote{We are indebted to
Prof. V. Bogachev, who nicely disproved our initial and naive guess
that perhaps even $C(1)=1$, i.e. $\int_{a-1}^{a+1} f  \le \int_{-1}^1
f\, (f\gg0)$ could hold. \\ \indent Actually, he took $f$ to be the
probability density function arising from of a convolution square of
some symmetric probability distribution. Then $f$ will be symmetric,
non-negative and positive definite. It is clearly possible that at some
point $L>0$ $f$ has $f'(L)>0$, whence by symmetry also $f(-L)=f(L)$ and
$f'(-L)=-f'(L)$. Now clearly if $\Phi(a):=\int_{a}^{a+2L} f$, then
$\Phi'(a)=f(a+2L)-f(a)$ and $\Phi"(a)=f'(a+2L)-f'(a)$, whence for a
small change $\delta>0$ we must have $\Phi(-L+\delta)-\Phi(-L)\approx
\Phi'(-L) \delta + \Phi"(-L) \delta^2/2 = f'(L)\delta^2 >0$, and so
with $a=-L+\delta$ the proposed inequality fails.} reads as follows.
\begin{theorem}\label{th:CG}
	For the extremal constant functions $G(\ell)$ and $C(\ell)$ the
following estimates hold. 	

1. {\bf Lower bound.} For any $\ell\in\RR\backslash\NN$, we have $G(\ell), C(\ell)\geq 2[\ell]+1$.

Moreover, for all $k\in\NN$, we have $G(k)\geq 2k-1$ and $C(k)\geq
2k$.

2. {\bf Upper bound.} For any $\ell\ge 1$, we have

\begin{equation}\label{CGupperbound}
G(\ell) \le C(\ell) \le \frac{1}{2}\frac{([2\ell]+1)([2\ell]+2)}{[2\ell]+1-\ell} \le \lceil 2 \ell \rceil +1.
\end{equation}

3. {\bf Sharpness.} As a consequence of the above, our bounds are exact
for $G(k+0)$ and $C(k+0)$, i.e., we have $\lim_{\ell\to k+0} G(\ell) =
\lim_{\ell\to k+0} C(\ell)= 2k+1$ for all $k\in \NN$.
\end{theorem}

We remark that instead of the space of doubly positive functions, we could consider the space of smooth doubly positive functions or only measurable doubly positive functions. However, the constants would not differ essentially.

The proof of Theorem \ref{th:CG} is composed of two lemmas. Before presenting them, let us explain the idea implemented in Lemma 1 since the actual formulas may hide it a little.

Our strategy is the following. We consider the so-called periodically
extended Dirac delta, that is, $\Phi := \sum_{k=-\infty}^{\infty}
\delta_{kp}$. This "function" is obviously non-negative and positive
definite since it can be regarded as the characteristic function of a
group, namely, the discrete group $p\ZZ$. Here the period $p$ is chosen
to be $1+\ve$ in order to minimize the presence of values $pz$ $(z \in
\ZZ)$ in the segment $[-1,1]$, but at the same time make these values
as densely occurring in other intervals as possible. It is easy to show
that $\int_{-1}^{1} \Phi \,dx=1$ and, for any given length $kp<\ell<
(k+1)p$, we will have $\int_{-\ell}^\ell \Phi \,dx= 2k+1$ and
$\int_{-0}^{\ell-0} \Phi \,dx=k+1$. The only technical matter is to
make this construction fitting into the class of doubly positive
functions. We will do this below.

\begin{lemma}\label{l:below}
For all $k\in \NN$ we have $G(k+0) \geq 2k+1$. Moreover, $C(k)\geq
2k$.
\end{lemma}
\begin{proof}
Let us fix $k\in \NN$ and $\ve>0$. We are to estimate $C(k)$ and
$G(k+\ve)$ from below.

Let $f_n(x):=\cos^{2n}(\frac{\pi}{p}x)$ where $1\le p <1.1$ and $n$ is assumed to be large enough. Clearly, for all values of the parameters $n$ and $p$ the function $f_n$ is doubly positive and $p$-periodic.
It is easy to see that with any given fixed value of $\de \in (0,0.1)$,
\[
\frac{\int_{-\de}^{\de}f_n(x)\,dx}{\int_{-p/2}^{p/2}f_n(x)\,dx}\rightarrow{}1 \quad (n\rightarrow +\infty),
\]
i.e., the function $f_n$ is concentrated in the segment $(-\de,\de)$
with respect to the period (in the limit, when $n\to\infty$). Thus, we
see that $f_n$ is concentrated in $\cup_{m\in\NN}\Omega_m$ where
$\Omega_m:=(-\de+mp,\de+mp)$.

To make estimates for $\int_{-k-\ve}^{k+\ve}f_n(x)\,dx$ for some
$0<\ve<1$, we have to find how many segments $\Omega_m$ are contained
in the intervals $[-k-\ve,k+\ve]$. If we chose $\de <p-1$, then the
interval $[-1,1]$ contains only one of $\Omega_m$, namely $\Omega_0$,
and is disjoint from the rest.

Now if we take $\de$ small enough, and $p$ sufficiently close to $1$
(more exactly, if $k(p-1)+\de<\ve$), then the interval $[-k-\ve,k+\ve]$
already contains all $\Omega_m$ with $-k\le m\le k$. Thus $G(k+\ve)\geq
(2k+1)$, whence also $G(k+0)\ge 2k+1$.

Furthermore, the inequality $C(k)\geq 2k$ can be easily seen from
considering e.g. the interval $[1,2k+1]$ which contains $\Omega_m$ for
all $m=1,\dots,2k$ whenever $2k(p-1)+\de <1$.
\end{proof}

As the functions $G(\ell)$ and $ C(\ell)$ are both non-decreasing, and
$C(\ell) \ge  G(\ell)$, we obtain the first lower bound of Theorem
\ref{th:CG}. As for $G(k)$, clearly we have $G(1)=1$, while for $k>1$ we can
easily use monotonicity of $G$ to derive $G(k)\ge G(k-1+0)\ge 2k-1$.
The other estimate $C(k)\ge 2k$ is contained in the above Lemma, whence
Part 1 of Theorem \ref{th:CG} is proved.

\begin{lemma}\label{l:above} We have for any $\ell>1$ the inequality
\begin{equation}\label{eq:Cabove}
C(\ell)\leq \frac{1}{2}\frac{([2\ell]+1)([2\ell]+2)}{[2\ell]+1-\ell}.
\end{equation}
\end{lemma}
\begin{proof}
Fix the interval $I:=[-1/2,1/2]$. For temporary use let us denote by
$\chi:=\chi_I$ the characteristic function of $I$, and we denote by
$\chi_a(x):=\chi(x-a)$ for indices $a\in \RR$. We shall also use the
triangle function $T:=\chi\star\chi=(1-|x|)_{+}$ (where $\xi_{+}:=\max
(\xi, 0)$) which is an important example of a non-negative positive
definite function.

Consider the functions $g_a :=\chi - \chi_a$ and $h_a:=g_a \star
\widetilde{g_a}$, where $\widetilde{g_a}(x):=\overline{g_a(-x)}$. So, in view of $g_a$
being a real-valued function, then
$\widetilde{g_a}(x)=g_a(-x)=\chi(x)-\chi_{-a}(x)$. Then obviously
$h_a(x)=T(x) -T(x+a) -T(x-a) +T(x)=2T(x) -(T(x+a)+T(x-a))$. Because $h$
is defined as a convolution square, it is obvious that $h_a\gg 0$ for
any $a\in\RR$.

Let us involve here an additional parameter $p$ with $0<p\leq1$.
Take now 
the sum 
$H_{a,k,p}:= \sum_{j=0}^k h_{a+j(2-p)}$ with some $k\in \NN$. Then we
can estimate $H_{a,k,p}$ the following way:
\begin{eqnarray*}
H_{a,k,p}(x)&:=&\sum_{j=0}^k h_{a+j(2-p)}(x) \\&=&2(k+1)T(x) - \left(\sum_{j=0}^k T(x+a+j(2-p)) ~+~\sum_{j=0}^kT(x-a-j(2-p))\right)
\\ & \le& 2(k+1)\chi_{[-1,1]}(x) - p (\chi_{[a-1+p,a+k(2-p)+1-p]}(x)+\chi_{[-a-k(2-p)-1+p,1-a-p]}(x)).
\end{eqnarray*}
Note that $H\gg 0$ together with its summands $h_{a+j(2-p)}$.
Multiplying by any (say, continuous) doubly positive function $f$ we
get by an application of Schur's theorem $Hf\gg 0$, whence
$$
0\le \widehat{Hf}(0)=\int_{-\infty}^\infty Hf \le 2(k+1) \int_{-1}^1 f(x)\,dx - 2p \int_{a+p-1}^ {a+k(2-p)+1-p} f(x)\,dx,
$$
using also that $f$, as a positive definite real-valued function, is
necessarily even.

Let now $b:=a+p-1$. It follows that
$$p \int_{b}^ {b+k(2-p)+2-2p} f(x)\, dx\leq (k+1) \int_{-1}^1 f(x)\, dx.$$
That is, in terms of $C(\ell)$,
$$C\left(\frac{(k+1)(2-p)-p}{2}\right)\leq\frac{k+1}{p}.$$
So let us take an arbitrary $\ell>1$, and write it in the form
$\ell=\frac{(k+1)(2-p)-p}{2}$ with suitable values of $k$ and $p$. This
is equivalent to $p=\frac{2(k+1)-2\ell}{k+2}$. Further, the double
inequality $0<p\leq 1$ is equivalent to $\ell-1< k\leq 2\ell$, whence
for $C(\ell)$ we obtain that the estimate
\begin{equation}\label{reveszefimov315}
C(\ell)\leq \frac{1}{2}\frac{(k+1)(k+2)}{k+1-\ell}
\end{equation}
holds true for any integer $k$ between $\ell$ and $2\ell$ (and with
the respective choice of parameter $p$). So, it remains to minimize
estimate (\ref{reveszefimov315}) in the range $k\in(\ell-1,2\ell]$ for
any fixed $\ell>1$.

In order to do so, consider the auxiliary functions
\[
\ffi:(\ell-1,2\ell]\to\RR,\quad \ffi(x):=\frac{(x+1)(x+2)}{x+1-\ell}=x+\ell+2+\frac{\ell(\ell+1)}{x+1-\ell}
\]
and
\[
\psi:(\ell,2\ell]\to\RR,\quad \psi(x):=\ffi(x)-\ffi(x-1).
\]
Straightforward computation gives us that $\psi(x)=0$ is equivalent to
\[
\frac{(x+1)(x+2)}{x+1-\ell}=\frac{x(x+1)}{(x-\ell)},
\]
the unique solution of which on the interval $(\ell,2\ell]$ being
$x=2\ell$. Since $\psi$ is continuous we deduce that the sign of the
function $\psi$ does not vary on the interior of its domain. Further,
as $\ell>1$, we find
\[
\psi(\ell+1)=\ffi(\ell+1)-\ffi(\ell)=1-\frac{\ell(\ell+1)}{2} <0.
\]
This yields that the function $\psi$ is negative in $(\ell,2\ell)$,
whence $\ffi$ is non-increasing on the set of integers in
$(\ell-1,2\ell]$. Therefore, the minimum of $\ffi$ is certainly
achieved at the unique integer in $(2\ell-1,2\ell]$, in other words at
$[2\ell]$. S, we substitute $k=[2\ell]$ in \eqref{reveszefimov315}
which indeed yields the desired inequality \eqref{eq:Cabove}.
\end{proof}


\bigskip

The last inequality in \eqref{CGupperbound} can be obtained easily
considering separately the cases where $2\ell=m\in \NN$ (providing
equality), and where $2\ell\not\in \NN$ (leading to strict inequality).

As in the above argument concerning the lower bound, the proof of the
upper estimate will be completed adding that $G(\ell)\le C(\ell)$,
always.

\bigskip
Finally, we are in a position to prove Part 3, that is, the
sharpness statement.

The inequality $G(k+0) \ge 2k+1$ is clear, because the lower estimate
in Part 1 provides $G(k+\ve)\ge [2(k+\ve)]+1 =2k+1$ for arbitrary
$\ve>0$. Moreover, from Lemma \ref{l:above} we also get
\[
\begin{gathered}
\displaystyle C(k+0)=\lim_{\ve\to+0} C(k+\ve) \le
\frac{1}{2}\frac{([2k+2\ve]+1)([2k+2\ve]+2)}{[2k+2\ve]+1-(k+\ve)} =\\
\lim_{\ve\to+0}\frac{1}{2}\frac{(2k+1)(2k+2)}{2k+1-(k+\ve)} \le 2k+1.
\end{gathered}
\]
Altogether, we have $2k+1\le G(k+0)\le C(k+0)\le 2k+1$ and equality holds everywhere,
as needed.

Therefore, the proof of Theorem \ref{th:CG} is complete.

\section{Concluding remarks}

Let us see what general framework for the construction of the above estimates and proofs can be set up. Basically, what we do is to look for an auxiliary function $H$, positive definite itself, and satisfying
\[
H \le A \chi_{[-1,1]}- B\chi_{[a,a+\ell]} - B \chi_{[-a-\ell,-a]}
\]
or
\[
H \le A \chi_{[-1,1]}- B\chi_{[-\ell,\ell]}.
\]
By Schur's theorem we also have that $fH\gg 0$ for any $f\gg 0$, whence
\[
0 \le \widehat{fH}(0) = \int_{-\infty}^\infty fH \le A \int_{-1}^1 f - 2B \int_{a}^{a+\ell} f
\]
and $C(\ell/2) \le A/2B$, or
\[
0 \le \widehat{fH}(0) = \int_{-\infty}^\infty fH \le A \int_{-1}^1 f - B \int_{-\ell}^{\ell} f
\]
and $G(\ell) \le A/B$.

Let $\ell >0$ be arbitrary. Let us now consider the extremal quantities
\[
\sigma(a,\ell):= \inf \left\{ \frac{A}{2B}~:~ \exists H\gg 0, H\le A \chi_{[-1,1]}- B\chi_{[a,a+\ell]}-B\chi_{[-a-\ell,-a]}  \right\}, ~~ \sigma(\ell):=\sup_{a\in \RR} \sigma(a,\ell),
\]
and
\[
\gamma(\ell):= 2 \sigma(0,\ell)=\inf \left\{ \frac{A}{B}~:~ \exists H\gg 0, H\le A \chi_{[-1,1]}- B\chi_{[-\ell,\ell]}  \right\}.
\]

Clearly, from the above it follows that we have $C(\ell/2) \le
\sigma(\ell)$ and $G(\ell)\le \gamma(\ell)$, always. Let us make a few
additional remarks here. First, the setting here is quite general, but
at least in $\RR$ any "reasonable" positive definite function $H$ can
be represented as a "convolution-square": $H=G\star\widetilde{G}$, of
say some $G\in L^2$ function, see e.g. \cite{EGR}. The construction of
$H:=H_{a,k,p}$ worked somehow along different lines, for we instead
represented $H$ as the sum of other convolution squares with
well-controlled supports: but in principle the direct convolution
square representation is also possible.

The above defined extremal problems $\sigma$ and $\gamma$ are very much
like the so-called Tur\'an or Delsarte extremal problems. The main
difference is that here we want to compare integrals over given
intervals to integrals over given central pieces, while in the Tur\'an
and Delsarte problems we normalize with respect to $f(0)$ and compare
to this normalization either \emph{the full integral}, or (in case of
the so-called "pointwise Tur\'an problem") \emph{a particular one-point
value}. In the recent work \cite{GT}, more concrete application of the
Tur\'an and Delsarte problems are worked out for the case of Wiener's
problem in several dimensions. This type of approach seems to be
reasonable here, too.

Since for $H\gg 0$ we necessarily have $0<\widehat{H}(0)=\int H$, it
follows immediately that $A \ge B\ell$, whence $\gamma(\ell)\ge \ell$
and $\sigma(\ell)\ge \ell/2$. But in the virtue of our lower
estimations of $C(\ell)$ and $G(\ell)$, it is apparent that these are
far from being sharp. From the other side, it could be well that we
would have $\sigma(\ell)=C(\ell/2)$ and $\gamma(\ell)=G(\ell)$. The
essential part of the above constructions (i.e. the ones for the upper
estimation) targeted the computation (or estimation) of $\gamma(\ell)$
and $\sigma(\ell)$. We conjecture that in principle this approach is
best possible.

The interested reader can consult for further details about the Tur\'an
and Delsarte problems and their applications in e.g. packing problems
in \cite{Co}, \cite{CE}, \cite{CKMRV}, \cite{Go1}, \cite{Ivanov},
\cite{Revesz2009}, \cite{TuranLCA} and \cite{Vi}.

\end{document}